\documentclass[12pt]{article}
\usepackage{graphicx,amsmath,amssymb,amsthm} % Required for inserting images

\usepackage{xlop}                               %
\usepackage{tikz}                               %
\usepackage{float}                              %
\usepackage{xcolor}                             %
\usepackage{pstricks-add, pst-eucl}             %
\usetikzlibrary{positioning,shapes,fit,arrows}

\newtheorem{theorem}{Theorem}
\newtheorem{lemma}[theorem]{Lemma}
\newtheorem{corollary}[theorem]{Corollary}
\newtheorem{definition}[theorem]{Definition}

\newcommand{\cysp}{\mathrm{CySp}}
\newcommand{\Id}{\mathrm{Id}}

\title{Cyclic Group Spectra for Some Small Relation Algebras}
\author{Jeremy F.~Alm, Ashlee Bostic, Claire Chenault,\\ Kenyon Coleman, Chesney Culver  \\  Department of Mathematics, Lamar University, Beaumont, TX, USA}
\date{March 2024}

\begin{document}

\maketitle

\begin{abstract}

    The question of characterizing the (finite) representable relation algebras in a ``nice" way  is open. The class $\mathbf{RRA}$ is known to be not finitely axiomatizable in first-order logic. Nevertheless, it is conjectured that ``almost all'' finite relation algebras are representable. 

    All finite relation algebras with three or fewer atoms are representable.  So one may ask, Over what cardinalities of sets are they representable?  This question was answered completely by Andr\'eka and Maddux  (``Representations for small relation algebras,'' \emph{Notre Dame J. Form. Log.}, \textbf{35} (1994)); they determine the spectrum of every finite relation algebra with three or fewer atoms.

    In the present paper, we restrict attention to cyclic group representations, and completely determine the cyclic group spectrum for all seven symmetric integral relation algebras on three atoms. We find that in some instances, the spectrum and cyclic spectrum agree; in other instances, the spectra disagree for finitely many $n$; finally, for other instances, the spectra disagree for infinitely many $n$. The proofs employ constructions, SAT solvers, and the probabilistic method.

\end{abstract}

% \end{document}

% \input{atom6-7}

\section{Introduction}

Given a finite relation algebra, it is undecidable whether that algebra has a representation \cite{HH}.  Yet most ``small'' algebras are representable -- for example, all finite algebras on three or fewer atoms are representable -- and it is conjectured that ``almost all'' finite algebras are representable (see for example the problem list in \cite{HHbook}). A recent result by Koussas \cite{Koussas} is that the class of atom structures of finite integral symmetry relation algebras obeys a 0-1 law, lending credence (in the first author's opinion, anyway) to this conjecture. The 2004 survey by Maddux  gives an interesting history of the early work on representation questions \cite{Madd04}.  

For finite relation algebras known to be representable, one may ask, ``Over what cardinalities of sets are they representable?'' For finite  algebras with three or fewer atoms,  this question was answered completely by Andr\'eka and Maddux  \cite{Andreka}. 

In the study of the combinatorial complexity of finite relation algebras, the central organizing questions are as follows:

\noindent Given a finite algebra $A$,

\begin{enumerate}
    \item Is $A$ representable?
    \item If so, is it representable over a finite set?
    \item If so, what is the set of integers $n$ such that $A$ is representable over a set of order $n$?
\end{enumerate}

In light of these questions, we define the \emph{spectrum} of an algebra as follows:

\begin{definition}
    If $A$ is a finite relation algebra, then 

    \[
    \mathrm{Spec}(A) = \{ n \in \mathbb{Z} :  A \text{ has a square representation over $n$ points} \}
    \]
\end{definition}

It is usually extremely difficult to determine the spectrum of an algebra. Indeed, even finding $\min (\mathrm{Spec}(A))$ can be very difficult.  For example, relation algebra $32_{65}$ (see \cite{Maddux} for the numbering system), with diversity atoms $a$, $b$, and $c$ with $bbb$ and $ccc$ forbidden, has been well-studied; it has been known to be finitely representable since 2008 \cite{AMM}. The representation given there was over  $416,714,805,914$ points, employing the probabilistic method.  By using the Lovasz Local Lemma, this was improved (using the same underlying probability space) by Dodd and Hirsch to $63,432,274,896$ \cite{DH}. By constructing a more ``efficient'' probability space, this was greatly improved by the first author and his undergraduate student; by using randomized computer search over said space, they found a representation over $2^{13} =8192$ points (unpublished). The smallest known representation of $32_{65}$ is over $2^{10} = 1024$ points, which was found by the first author by combining randomized methods with painstaking \emph{ad hoc} manipulation \cite{Almetal}. The best we know as of this writing is as follows:

\[
26 \leq \min (\mathrm{Spec}(32_{65})) \leq 1024
\]

The lower bound was achieved by combining the use of a SAT solver with clever combinatorial insights provided by the graduate student coauthors on \cite{Almetal}.  

As the reader can see, the gap between 26 and 1024 is quite large, and is reminiscent of such lower-bound-upper-bound gaps in Ramsey theory. We will probably never know the true value of $\min (\mathrm{Spec}(32_{65}))$, let alone determining $\mathrm{Spec}(32_{65})$ completely.  These are challenging questions. 

Relation algebra $32_{65}$ admits a natural generalization, as follows. Define $A_n$ to be the finite symmetric integral algebra with diversity atoms $a$, $b_1$, \ldots, $b_n$ with all and only the ``all blue'' cycles $b_ib_jb_k$ forbidden. Define 

\[
    f(n) = \min\mathrm{Spec}(A_n).
\]
In \cite{AMM}, it was shown that $f(n) \leq \binom{15n^2}{n}$. Dodd and Hirsch improve this result  somewhat but do not give an explicit bound. Recently, $f(n)$ was shown to have polynomial growth \cite{Almetal}; in particular, we have that 
\[
    2n^{2} + 6n + 2 \leq f(n) \leq 2n^{6 + o(1)}.
\]

Some non-symmetric generalizations of $32_{65}$ were studied in \cite{Alm23}, leveraging a method originated by Comer \cite{comer83} and amplified in \cite{Alm23b}.   

In order to chip away at such intractable questions, it is helpful to narrow the scope of one's search. Several known representations of minimal size of ``small'' algebras  are over cyclic groups, so one might restrict one's attention to cyclic representations.  Indeed, that is the perspective of the present paper. We make the following definition.

\begin{definition}
    If $A$ is a finite relation algebra, then the cyclic spectrum of $A$ is
    \[
    \cysp(A) = \{ n \in \mathbb{Z} :  A \text{ has a  representation over } \mathbb{Z}/n\mathbb{Z} \}
    \]
\end{definition}

In this paper, we completely determine the cyclic group spectrum for all seven symmetric integral relation algebras on three atoms.  See Section \ref{sec:summary} for a summary of results. 

The algebras in this paper are named $1_7$, $2_7$,\ldots, $7_7$ by Maddux \cite{Maddux}.  Each has symmetric atoms $1'$, $a$, and $b$, and the distinct diversity cycles (up to equivalence) are $aaa$, $bbb$, $abb$, and $baa$. 

The seven algebras with mandatory diversity cycles are displayed in Table \ref{tab:cycles}.

\begin{table}[hbt]
    \centering
    \begin{tabular}{c|cccc}
          $1_7$  &  &  & $abb$ &  \\
          $2_7$  & $aaa$ &  & $abb$ &  \\
          $3_7$  &  & $bbb$ & $abb$ &  \\
          $4_7$  & $aaa$ & $bbb$ & $abb$ &  \\
          $5_7$  &  &  & $abb$ & $baa$ \\
          $6_7$  & $aaa$ &  & $abb$ & $baa$ \\
          $7_7$  & $aaa$ & $bbb$ & $abb$ & $baa$ \\
    \end{tabular}
    \caption{Mandatory diversity cycles for $1_7$ through $7_7$}
    \label{tab:cycles} 
\end{table}

\section{Relation algebra $1_7$}%%%%%%%%%%%%%%%%%%%%%%%%%%%%%%%%%%%%%%%%%%%%%%%%%%%%%%%%%%%%%%%%%%%%%%%%%%%%%%%%%%%%%%%%%%%%%%%%%%%%%%%%%%
\begin{theorem}
    Relation algebra $1_7$ is  representable over $\mathbb{Z}/n\mathbb{Z}$ for  $n=4$, and $\cysp (1_7) = \{4\}$.
\end{theorem}

It is known that the spectrum of $1_7$ is $\{4\}$. Thus it suffices to exhibit a representation over $\mathbb{Z}/4\mathbb{Z}$.

Since $Id \cup A$ is a subgroup, we must have $A = \{2\}$. This forces $B = \{1,3\}$. Then $B+B = Id \cup A$ and $A+B = B$, as required.

\section{Relation algebra $2_7$}%%%%%%%%%%%%%%%%%%%%%%%%%%%%%%%%%%%%%%%%%%%%%%%%%%%%%%%%%%%%%%%%%%%%%%%%%%%%%%%%%%%%%%%%%%%%%%%%%%%%%%%%%%
\begin{theorem}\label{thm:2_7}
    Relation algebra $2_7$ is representable over $\mathbb{Z}/n\mathbb{Z}$ for even $n\geq 6$, and $\cysp (2_7) = \{2k : k\geq 3\}$.
\end{theorem}

\begin{lemma}\label{lem:sum}
    Let $n$ be odd. Let $S$ be a sum-free subset of $\mathbb{Z}/n\mathbb{Z}$. Then $S$ cannot contain more than $n/2$ elements.
\end{lemma}

\begin{proof}
    Suppose for contradiction that $S\subseteq \mathbb{Z}/n\mathbb{Z}$ is sum-free and contains more than $n/2$ elements.  Consider pigeonholes labelled 0 through $n-1$ (the elements of $\mathbb{Z}/n\mathbb{Z}$).  For each element in $S$, place a pigeon in the hole labelled with that element.  This leaves fewer than half the holes unoccupied. 
    
    Now, for each $x\in S$, place a pigeon in the hole labelled $x+x\pmod{n}$.  Since $S$ is sum-free, no pigeon is placed in a hole occupied by a pigeon from the previous paragraph, so some hole gets two ``$x+x$'' pigeons, so there are distinct $x,y\in\mathbb{Z}/n\mathbb{Z}$ with $x+x=y+y$. But  $n$ is odd, so the map $a \mapsto a+a$ is a bijection on $\mathbb{Z}/n\mathbb{Z}$, which is a contradiction. 
\end{proof}

\begin{proof}[Proof of Theorem \ref{thm:2_7}]
    We need to partition $\mathbb{Z}/n\mathbb{Z}$ Into $\Id$, $A$, and $B$ such that 

\begin{enumerate}
    \item $A+A=A\cup \Id$
    \item $B+B=A\cup \Id$
    \item $A+B=B$
\end{enumerate}

If $n$ is even, then we can let $A =\{ 2, 4, \ldots, n-2\}$ and $B=\{1,3,5,\ldots, n-1\}$, i.e., the even and odd residues. (We will abuse terminology and treat the residues modulo $n$  as ``regular integers,'' not equivalence classes.) 

To see that if $n$ is odd, a representation over $\mathbb{Z}/n\mathbb{Z}$ does not exist, consider $n=2k+1$. We have that $A\cup \Id$ is a subgroup of $\mathbb{Z}/n\mathbb{Z}$, and since 2 does not divide $n$, we conclude that $|A\cup \Id|\leq n/3$. Thus $|B|\geq 2n/3$, so by Lemma \ref{lem:sum}, $B$ is not sum-free, a contradiction. 
\end{proof}

\section{Relation algebra $3_7$}%%%%%%%%%%%%%%%%%%%%%%%%%%%%%%%%%%%%%%%%%%%%%%%%%%%%%%%%%%%%%%%%%%%%%%%%%%%%%%%%%%%%%%%%%%%%%%%%%%%%%%%%%%
\begin{theorem}
    Relation algebra $3_7$ is representable over $\mathbb{Z}/n\mathbb{Z}$ for even $n\geq 6$, and $\cysp (3_7) = \{2k : k\geq 3\}$.
\end{theorem}

\begin{proof}
    Since the spectrum of $3_7$ is $\{2k : k\geq 3\}$, it suffices to exhibit cyclic representations for even $n\geq 6$. 

    We need to partition $\mathbb{Z}/n\mathbb{Z}$ into $\Id$, $A$, and $B$ such that 

\begin{enumerate}
    \item $A+A=A\cup \Id$
    \item $B+B= \Id$
    \item $A+B=A\cup B$
\end{enumerate}

The second condition tells us that  $B$ consists of a single self-inverse element. Setting $B=\{n/2\}$ and $A = \{1,2,\ldots, n-1\}\setminus\{n/2\}$ suffices.

\end{proof}

\section{Relation algebra $4_7$}%%%%%%%%%%%%%%%%%%%%%%%%%%%%%%%%%%%%%%%%%%%%%%%%%%%%%%%%%%%%%%%%%%%%%%%%%%%%%%%%%%%%%%%%%%%%%%%%%%%%%%%%%%
\begin{theorem}
    Relation algebra $4_7$ is representable over $\mathbb{Z}/n\mathbb{Z}$ for composite $n>8$ such that $n$ is not equal to 2 times a prime, and $\cysp (4_7) = \{ n>8 : \ n\text{ is composite, } n \text{ is not twice a prime}\}$. 
\end{theorem}

\begin{proof}
    First, we establish the claim that $A\cup I$ must be a subgroup of $\mathbb{Z}/n\mathbb{Z}$. Since $aaa$ is mandatory, $A\subseteq A+A$.  Since $a$ and $b$ are the only diversity atoms and $baa$ is forbidden, $A+A = A \cup I$, so $A\cup I$ is closed under addition. Since it is also closed under subtraction, $A\cup I$ is a subgroup.

    Next, we show that a necessary condition for representability over $\mathbb{Z}/n\mathbb{Z}$ is that $n$ have a divisor $k>2$ such that $n/k>2$, hence a subgroup $H$ of order greater than two and index greater than two. First, note that if $n$ is prime, we have no proper nontrivial subgroups. 
    
    So suppose for contradiction that $4_7$ is representable over $\mathbb{Z}/n\mathbb{Z}$, but the only divisors of $n$ are 1, 2, $n/2$, and $n$.  If $|H|=2$ and $A = H\setminus\{0\}$, then $n$ is even, $A = \{n/2\}$, and $A+A=I$, a contradiction.  If $|H| = n/2$, then again $n$ is even and $H$ consists of all the even numbers less than $n$. This means that $B$ consists of all the odds. But that makes $B$ sum-free, which contradicts the fact that $bbb$ is mandatory.    

    Now we show that $n$ having a divisor $k>2$ such that $n/k>2$ is sufficent for representability over $\mathbb{Z}/n\mathbb{Z}$. Let $H$ be the unique subgroup of order $k$, let $A = H\setminus\{0\}$, and let $B=(\mathbb{Z}/n\mathbb{Z})\setminus H$. The fact that $A+A=I\cup A$ is immediate.  To see that $A+B=B$, let $x\in B$. Choose any $y\in A$. Then $z=x-y$ is clearly in $B$, so $y+z=x$ and $y$ and $z$ witness that $x\in A+B$. $A$ is disjoint from $A+B$ because otherwise $A\cup I$  is not closed under the group operations. 

    To see that $B+B=I\cup A \cup B$, let $x\in A$, and consider the various cosets of $H$ in $\mathbb{Z}/n\mathbb{Z}$. Because $H$ is the identity coset, there exist $y,z\notin H$ such that $x=y+z$. Note that $y,z\in B$.  

    Now let $x\in B$. Suppose $x$ is some coset $gH$ (which cannot be $H$). Because $H$ has index greater than two, we can find $h,k\notin H$ such that $gH = hH \oplus kH$, where $\oplus$ is the operation in the quotient group. (We are using juxtaposition for cosets, to avoid confusion with set-wise addition.) Therefore there are $h',k'\in B$ such that $x=h'+k'$, as desired. 
    
\end{proof}
Note that we used the fact that the underlying group was cyclic only to establish the existence of a subgroup of the desired order and index.  Hence we have the following corollary.
\begin{corollary}
    Let $G$ be a finite abelian group.  Then $4_7$ is representable over $G$ if and only if $G$ has a subgroup $H$ of order greater than 2 and of index greater than 2. 
\end{corollary}

\section{Relation algebra $5_7$}%%%%%%%%%%%%%%%%%%%%%%%%%%%%%%%%%%%%%%%%%%%%%%%%%%%%%%%%%%%%%%%%%%%%%%%%%%%%%%%%%%%%%%%%%%%%%%%%%%%%%%%%%%
\begin{theorem}
    Relation algebra $5_7$ is representable over $\mathbb{Z}/5\mathbb{Z}$, and $\cysp (5_7) = \{5\}$.
\end{theorem}

\begin{proof}
    The fact that $\mathrm{Spec}(5_7) = \{5\}$ is well-known folklore. Note that the 1-cycles are forbidden in both colors. The connection to the Ramsey number $R(3,3)=6$, which implies that any edge-coloring in two colors on a complete graph on six or more vertices must contain a monochromatic triangle, hence forbidding a representation, has been noted often.  

    Relation algebra $5_7$ is representable over $\mathbb{Z}/5\mathbb{Z}$ by setting $A=\{1,4\}$ and $B=\{2,3\}$. 
\end{proof}

\section{Relation algebra $6_7$}%%%%%%%%%%%%%%%%%%%%%%%%%%%%%%%%%%%%%%%%%%%%%%%%%%%%%%%%%%%%%%%%%%%%%%%%%%%%%%%%%%%%%%%%%%%%%%%%%%%%%%%%%%
\begin{theorem}
    Relation algebra $6_7$ is representable over $\mathbb{Z}/n\mathbb{Z}$ for $n=8$, $n \geq 11$ , and $\cysp (6_7) = \{8\}\cup\{n : n\geq 11\}$.
\end{theorem}

\begin{proof}
The representation over $\mathbb{Z}/8\mathbb{Z}$ is well-known and is given by $A=\{2,3,5,6\}$, $B=\{1,4,7\}$. 

      For $n\geq 11$, we need to partition $\mathbb{Z}/n\mathbb{Z}$ into $\Id$, $A$, and $B$ such that 

\begin{enumerate}
    \item $A+A=A\cup B \cup \Id$
    \item $B+B= A\cup \Id$
    \item $A+B=A\cup B$
\end{enumerate}

In this case, $B$ is what is called a ``symmetric complete sum-free set'' in the additive number theory literature.  Many examples are well-known and we quote them without proof.

For the case $n=3k+2$, the construction is quite straightforward. One sets $B$ to be the interval $[k+1, 2k+1]$. So $B$ is the ``middle third.''  See, for example, \cite{Levy}. 

For the case $n=3k+1$, we modify the middle third  construction slightly. One sets $B$ to be the union  $\{k\} \cup [k+2, 2k-1] \cup \{2k+1\}$.   See  \cite{Levy} again. 

We were not able to find a construction in the literature for the case $n=3k$. We give a construction here, but we doubt its originality.  

We split into two cases, $n=6k$ and $n=6k+3$. For $n=6k$, set $B = [k, 2k-1] \cup [4k+1, 5k]$. For $n=6k+3$, set $B = [k, 2k-1] \cup [4k+4, 5k+3]$. It is somewhat tedious but straightforward to verify that these sets work. 

Finally, we need to rule out cyclic representations for $n= 9, 10$. For $n=9$, we have $|A|, |B| \in {2,4,6}$. Neither can have order 2, so we must have $|A|=|B|=4$. A bit of case analysis shows there is no symmetric complete sum-free set of order 4 in $\mathbb{Z}/9\mathbb{Z}$. In addition, the non-existence of cyclic representations for $n= 9, 10$ was confirmed via SAT solver.

\end{proof}

\section{Relation algebra $7_7$}%%%%%%%%%%%%%%%%%%%%%%%%%%%%%%%%%%%%%%%%%%%%%%%%%%%%%%%%%%%%%%%%%%%%%%%%%%%%%%%%%%%%%%%%%%%%%%%%%%%%%%%%%%
\begin{theorem}
    Relation algebra $7_7$ is representable over $\mathbb{Z}/n\mathbb{Z}$ for  $n\geq 12$, and $\cysp (7_7) = \{n : n\geq 12\}$.
\end{theorem}

\begin{proof}

 To find a cycle representation, we need to partition $\mathbb{Z}/n\mathbb{Z}$ into $\Id$, $A$, and $B$ such that 

\begin{enumerate}
    \item $A+A=A\cup B \cup \Id$
    \item $B+B= A\cup B \cup\Id$
    \item $A+B=A\cup B$
\end{enumerate}
    First we give a probabilistic argument. For any nonzero $z\in\mathbb{Z}/n\mathbb{Z}$, there are at least $(n-2)/2$ sets $\{x,y\}$ such that $x+y=z$. Each $z$ has four ``needs'', namely $aa$, $ab$, $ba$, and $bb$. For each need $c_1c_2$, there must be some $x$, $y$ with $x$ labelled $c_1$, $y$ labelled $c_2$, and $x+y=z$. Note that because $A+B=B+A$, the $ab$ need is satisfied if and only if the $ba$ need is, so we will treat the needs as being simply $aa$, $ab$, and $bb$. 

    Given $z$ and a need $c_1c_2$, the probability that $x$ and $y$ are not labeled $c_1$ and $c_2$, respectively, is $3/4$. Hence the probability that $z$ does not have need $c_1c_2$ satisfied is at most $(3/4)^{(n-2)/2}$. Therefore the probability that $z$ has \emph{some} need unsatified is bounded above by $3\cdot (3/4)^{(n-2)/2}$ (because there are three needs).  Hence, by the union bound, the probability there there is some $z$ with some unsatisfied need is bounded above by 
    \[
        3(n-1)(3/4)^{(n-2)/2},
    \]
    which quantity we want to be less than 1. Throwing the above expression into Excel tells us that $n\geq 34$ is necessary and sufficient. 

    Now it remains to check all $9 \leq n \leq 33$, which we do using the SAT solver CryptoMiniSAT via the python library \texttt{boolexpr}.  Checking all  $9 \leq n \leq 33$, we see that $7_7$ is representable over $\mathbb{Z}/n\mathbb{Z}$ for $12 \leq n \leq 33$. This completes the proof. 
\end{proof}

\section{Summary}\label{sec:summary}

\begin{table}[H]
    \centering
    \begin{tabular}{r|c|c}
         & Spec & Cyclic Spec\\
     \hline
    $1_7$ & $\{4\}$ & $\{4\}$\\
    $2_7$ & $\{n \geq 6\}$ & $\{2k : k \geq 3\}$\\
    $3_7$ & $\{2k : k \geq 3\}$ & $\{2k : k \geq 3\}$\\
    $4_7$ & $\{ n \geq 9\}$ & $\{n \geq 9\}\setminus\{p, 2p : p \text{ prime} \}$\\
    $5_7$ & $\{5\}$ & $\{5\}$\\
    $6_7$ & $\{n \geq 8\}$ & $\{8\} \cup \{ n \geq 11 \}$\\
    $7_7$ & $\{ n \geq 9\}$ & $\{n \geq 12\}$\\
    
    \end{tabular}
    \caption{Summary of Results}
    \label{tab:summary}
\end{table}

\section*{Acknowledgements}

The first author wishes to thank Peter Jipsen for encouraging him, at the 2007 Logic meeting at Vanderbilt when he was a fresh PhD, to find the smallest representation of $32_{65}$. He is still working on it.

% \begin{tabular}{r|c|c}
%      & Spec & Cyclic Spec\\
%      \hline
%     $1_7$ & $\{4\}$ & $\{4\}$\\
%     $2_7$ & $\{n \geq 6\}$ & $\{2k : k \geq 3\}$\\
%     $3_7$ & $\{2k : k \geq 3\}$ & $\{2k : k \geq 3\}$\\
%     $4_7$ & $\{ n \geq 9\}$ & $\{n \geq 9\}\setminus\{p, 2p : p \text{ prime} \}$\\
%     $5_7$ & $\{5\}$ & $\{5\}$\\
%     $6_7$ & $\{n \geq 8\}$ & $\{ n \geq 11 \}$\\
%     $7_7$ & $\{ n \geq 9\}$ & $\{n \geq 12\}$\\
    
% \end{tabular}

\bibliographystyle{plain}
\bibliography{refs}

\end{document}